\documentclass[12pt]{amsart}
\usepackage{amscd,amssymb}

\newtheorem{theorem}{Theorem}[section]
\newtheorem{lemma}[theorem]{Lemma}
\newtheorem{proposition}[theorem]{Proposition}
\newtheorem{corollary}[theorem]{Corollary}

\newtheorem{problem}[theorem]{Problem}
\newtheorem{conjecture}[theorem]{Conjecture}

\theoremstyle{definition}

\newtheorem{remark}[theorem]{Remark}

\begin{document}

\title[Central polynomials of minimal degree for matrices]
{Central polynomials of minimal degree for matrices}

\author[Vesselin Drensky, Boyan Kostadinov]
{Vesselin Drensky, Boyan Kostadinov}
\address{Institute of Mathematics and Informatics
\newline Bulgarian Academy of Sciences, 1113 Sofia, Bulgaria}
\email{drensky@math.bas.bg, boyan.sv.kostadinov@gmail.com}

\thanks{The research of the second named author was supported
by the National Programme of Ministry of Education and Science
``Young Scientists and Post-Doctoral Researchers - 2''.}


\maketitle

\begin{abstract}
Formanek made the conjecture that the minimal degree of the central polynomials for the $n\times n$ matrix algebra over a field of characteristic 0 is $(n^2+3n-2)/2$
and this is true for $n\leq 3$. For $n=4$ there are examples of central polynomials of degree $13=(4^2+3\cdot 4-2)/2$ and we do not know whether there are central polynomials of lower degree.
In this paper we discuss methods for searching for central polynomials of low degree and prove
that the algebra of $4\times 4$ matrices does not have central polynomials in two variables of degree $\leq 12$.
As a byproduct of our computations we obtain that this algebra does not have also polynomial identities in two variables of degree $\leq 12$.
\end{abstract}

\section{Introduction}

Let $K\langle X\rangle=K\langle x_1,x_2,\ldots\rangle$ be the free associative algebra over a field $K$, i.e. the algebra of polynomials in noncommuting variables.
The polynomial $c(x_1,\ldots,x_d)\in K\langle X\rangle$ is a central polynomial for a $K$-algebra $R$
if $c(r_1,\ldots,r_d)$ is in the center of $R$ for all $r_1,\ldots,r_d\in R$ but it is not a polynomial identity for $R$, i.e. it does not vanish evaluated on $R$.
For example, by the Cayley-Hamilton theorem for the $2\times 2$ matrix algebra $M_2(K)$
\[
a^2-\text{tr}(a)a+\det(a)I_2=0,\,a\in M_2(K),
\]
where $I_2$ is the identity matrix. Since $\text{tr}(ab)=\text{tr}(ba)$, $a,b\in M_2(K)$, we obtain that
$\text{tr}([a,b])=0$, where $[a,b]=ab-ba$ is the commutator of $a,b$. Hence
\[
[a,b]^2=-\det([a,b])I_2
\]
and the polynomial
\[
c(x_1,x_2)=[x_1,x_2]^2
\]
is a central polynomial for the $2\times 2$ matrix algebra $M_2(K)$.

In 1956 Kaplansky \cite{Kap1} gave a talk where he asked 12 problems which
motivated significant research activity in the next decades. Problem 5 of
his problems was about the existence of central polynomials for matrices.

\begin{problem}\label{Problem of Kaplansky}
Does there exist a polynomial $c(x_1,\ldots,x_d)\in K\langle X\rangle$ which always takes values in the center of $M_n(K)$
without being identically $0$?
\end{problem}

According to Kaplansky, Herstein informed him that over a finite field and for any $n$ John Thompson constructed a polynomial in one variable which sends
any $n\times n$ matrix to a scalar, and not always to the same scalar.
A nonconstructive proof that $n\times n$ matrices over a finite field have central polynomials was given by Latyshev and Shmelkin \cite{LaSh}.

Later Kaplansky \cite{Kap2} revised his problems and the new version of the problem became:

\begin{problem}\label{Revised problem of Kaplansky}\cite[Problem 16]{Kap2}
Does there exist a homogeneous multilinear polynomial of positive degree which always takes values in the center of $M_n(K)$, $n\geq 3$, without being identically $0$?
\end{problem}

In this form the problem also eliminates the trivial answer when
\[
c(x_1,\ldots,x_{2n})=1+s_{2n}(x_1,\ldots,x_{2n})
\]
because by the Amitsur-Levitzky theorem \cite{AmLev} the standard identity $s_{2n}=0$ is a polynomial identity for $M_n(K)$.

The answer to the problem of Kaplansky was given in 1972-1973
by Formanek \cite{Fo1} and Razmyslov \cite{Raz1}. Later, in 1979, an alternative one-page proof
of the existence of central polynomials
based on results of Amitsur was given by Kharchenko \cite{Kh}.

The existence of central polynomials for matrices was very important for ring theory.
Very soon important theorems were established or simplified using central polynomials, see the monograph of Jacobson \cite{Ja}.

The central polynomial of Formanek for $M_n(K)$ is of degree $n^2$
and the polynomial of Razmyslov is of degree $3n^2-1$.
Using the method of Razmyslov, Halpin \cite{Ha} constructed a central polynomial which is also of degree $n^2$.
There was a common belief that the minimal degree of the central polynomials of $M_n(K)$ is $n^2$
and this is true for $n\leq 2$.

If not explicitly stated, in what follows we assume that the base field $K$ is of characteristic 0.
In 1984 Drensky and Kasparian constructed a central polynomial of degree 8 for $M_3(K)$ \cite{DrKas2}
and showed that $M_3(K)$ does not have central polynomials of lower degree \cite{DrKas1}.
These and other new results were included in the master thesis of Kasparian \cite{Kas} (with supervisor Drensky).
The new central polynomial of degree 8 was obtained using ideas of the Rosset proof \cite{Ros} of the Amitsur-Levitzki theorem.
(The proof of Rosset uses matrices with entries from the Grassmann algebra.)
The approach to show that there are no central polynomials of degree 7
was based on computations by hand combined with techniques of representation theory of the general linear group.
With the same method more central polynomials of degree 8 were found.
Using a computer, Bondari \cite{Bo} found all central polynomials of degree 8 for $M_3(K)$.
Three of them were known from \cite{DrKas1} and one of them was new.
The following conjecture is due to Formanek \cite{Fo2}.

\begin{conjecture}\label{Conjecture of Formanek}
The minimal degree
of the central polynomials for $M_n(K)$ over a field
$K$ of characteristic 0 is equal to
\[
\text{\rm mindeg}(M_n(K))=\frac{1}{2}(n^2+3n-2).
\]
\end{conjecture}

Although this is not a reason, the only quadratic function $p(n)$ with
$p(1)=1$, $p(2)=4$ and $p(3)=8$ is $(n^2+3n-2)/2$.
What is more important, the conjecture
agrees with other conjectures in the theory of
PI-algebras.

Drensky and Rashkova \cite{DrRa} studied weak polynomial identities of matrices.
(The knowledge of weak polynomial identities is a key moment of the construction of central polynomials by Razmyslov \cite{Raz1}.)
In particular the obtained by computer results explained the existence of a central polynomial of degree 8 for $3\times 3$ matrices.

For $n=4$ Drensky and Piacentini Cattaneo \cite{DrPC} found
a new central polynomial of degree 13. Pay attention:
\[
\frac{1}{2}(4^2+3\cdot 4-2)=13.
\]
Unfortunately we do not know whether $M_4(K)$ has central polynomials of degree 12.
The construction uses a weak polynomial identity of degree 9 found by computer
(with methods similar to those of Drensky and Rashkova \cite{DrRa})
and combines the methods of Formanek \cite{Fo1} and Razmyslov \cite{Raz1}.
The result was generalized by Drensky \cite{Dr2} to construct central
polynomials of degree $(n-1)^2+4$ for $M_n(K)$, $n>2$.
(For $n=3$ and $n=4$ this agrees with Conjecture \ref{Conjecture of Formanek}.

There are many other proofs for the existence of central polynomials for the $n\times n$ matrix algebra.
Giambruno and Valenti \cite{GiVa} constructed central polynomials using techniques from invariant theory of matrices.

The proof of Kharchenko \cite{Kh} from 1979 is not constructive. There are several other nonconstructive proofs which work also over a field of positive characteristic.
The main ideas are similar. As the proof of Kharchenko, the proof of Braun \cite{Bra} in 1982
used old results by Amitsur from the 1950's,
before Kaplansky stated his problem. The more recent proofs of Bre\v{s}ar \cite{Bre1, Bre2}
are more self-contained and use techniques typical
for generalized polynomial identities and functional identities.

Bre\v{s}ar and Drensky \cite{BreDr} showed that starting with a multihomogeneous central polynomial for $M_n(K)$ when $K$ is an infinite field of positive characteristic $p$
one can produce a polynomial of the same degree with coefficients in the prime field ${\mathbb F}_p$ which is central for the $n\times n$ matrices over any field $F$ of characteristic $p$.
The proof is elementary and uses only standard combinatorial techniques. It also completes the nonconstructive proofs removing the requirement of the infinity of the field $K$ in some of them.

In this paper we discuss computational methods to determine whether the algebra $M_n(K)$ has central polynomials of a given degree $m$.
In more detail we consider the method which is similar to the method in \cite{DrKas1}
used to determine the polynomial identities of degree 8 form $M_3(K)$.

If we want to confirm the conjecture of Formanek we can use additional arguments to simplify the computations.
The first step in this direction is to show that there are no central polynomials of degree 12 for $M_4(K)$.
Applying our algorithm we show that over a field $K$ of characteristic 0 the algebra $M_4(K)$ has no central polynomials of degree $\leq 12$ in two variables.
Since the method works also for the study of the polynomial identities of $M_n(K)$,
as a byproduct of our computations we obtain that $M_4(K)$ does not have also polynomial identities in two variables of degree $\leq 12$.
All computations are performed on usual personal computers and laptops using freely available software.

Concluding the introduction, the class of algebras with polynomial identities, i.e. PI-algebras,
is sufficiently big and contains important classes of algebras (e.g. all commutative and all finite dimensional algebras).
On the other hand it is reasonably small and has rich structure and combinatorial theory.
The theory of polynomial identities and central polynomials for matrices is an important part of the theory of PI-algebras.
It is closely related with structure and combinatorial ring theory, commutative and noncommutative invariant theory and other branches of algebra, see the books
listened in the end of the introduction. Below we shall discuss one more actively studied problem, also related with the topic of our paper:
the description of the images of polynomials on noncommutative algebras. The following conjecture was originally posed by Kaplansky
and later formulated by L'vov \cite[Question 1.98]{DnN}.

\begin{conjecture}\label{L'vov-Kaplansky}
Let $p$ be a multilinear polynomial. Then the set of values
of $p$ on the matrix algebra $M_n(K)$ over an infinite field $K$ is a vector space.
\end{conjecture}

It is easy to see that this conjecture is equivalent to the following.

\begin{conjecture}\label{L'vov-Kaplansky revised}
If $p$ is a multilinear polynomial evaluated on the matrix algebra $M_n(K)$, then
$\text{\rm Im}(p)$ is either $\{0\}$, $K$, $sl_n(K)$, or $M_n(K)$. Here $K$ is identified with the set of scalar matrices
and $sl_n(K)$ is the set of matrices of trace zero.
\end{conjecture}

Conjecture \ref{L'vov-Kaplansky revised} is confirmed by Kanel-Belov, Malev and Rowen \cite{K-BMRow}
only for $2\times 2$ matrices over a quadratically closed field $K$ (i.e. for fields $K$ which contain the zeros of all
nonconstant polynomials $f(x)\in K[x]$ of degree 2).

\begin{theorem}\cite[Theorem 2]{K-BMRow}
If $p$ is a multilinear polynomial evaluated on the matrix algebra $M_2(K)$ over a quadratically closed field $K$, then $\text{\rm Im}(p)$ is either
$\{0\}$, $K$, $sl_2$, or $M_2(K)$.
\end{theorem}

A survey on Conjecture \ref{L'vov-Kaplansky revised} is given by Kanel-Belov, Malev, Rowen and Yavich \cite{K-BMRowY}
and the images of polynomials on other noncommutative algebras were studied by
Kanel-Belov, Malev, Pines and Rowen \cite{K-BMPRow}. See also the very recent paper by Vitas \cite{V}.

For further reading and the necessary information needed for our paper we recommend some of the books (listed chronologically)
by Procesi \cite{Pro}, Jacobson \cite{Ja}, Rowen \cite{Row1, Row2, Row3}, Razmyslov \cite{Raz2}, Kemer \cite{Kem}, Drensky \cite{Dr3}, Drensky and Formanek \cite{DrFo},
Giambruno and Zaicev \cite{GiZa}, Kanel-Belov and Rowen \cite{K-BRow}, Kanel-Belov, Karasik and Rowen \cite{K-BRowKar},
Aljadeff, Giambruno, Procesi and Regev \cite{AlGiProReg}.

\section{Preliminaries}
Recall that a polynomial $f(x_1,\ldots,x_d)\in K\langle X\rangle$ is a polynomial identity for the algebra $R$ if
\[
f(r_1,\ldots,r_d)=0\text{ for all }r_1,\ldots,r_d\in R.
\]
We denote by $T(R)$ the ideal of all polynomial identities of $R$. Such ideals are invariant under all endomorphisms of $K\langle X\rangle$ and are called T-ideals.
In other words, the ideal $I$ of $K\langle X\rangle$ is a T-ideal if and only if for any $f(x_1,\ldots,x_d)\in I$ and any $u_1(X),\ldots,u_d(X)\in K\langle X\rangle$
we have $f(u_1(X),\ldots,u_d(X))\in I$. We also denote by $\langle f\rangle^T$ the T-ideal of $K\langle X\rangle$ generated by the polynomial $f$.

The $d$-generated subalgebra $K\langle X_d\rangle=K\langle x_1,\ldots,x_d\rangle$ of $K\langle X\rangle$ is graded
\[
K\langle X_d\rangle=\bigoplus_{m\geq 0}K\langle X_d\rangle^{(m)},
\]
where $K\langle X_d\rangle^{(m)}$ consists of all homogeneous polynomials of degree $m$.
This algebra is also multigraded
\[
K\langle X_d\rangle=\bigoplus_{m_i\geq 0}K\langle X_d\rangle^{(m_1,\ldots,m_d)},
\]
where $K\langle X_d\rangle^{(m_1,\ldots,m_d)}$ consists of all polynomials which are homogeneous of degree $m_i$ in $x_i$, $i=1,\ldots,d$.
We shall denote by $P_m$ the multilinear component $K\langle x_1,\ldots,x_m\rangle^{(1,\ldots,1)}$ of degree $m$ of $K\langle x_1,\ldots,x_m\rangle$.

Since we work over a field $K$ of characteristic 0, if
\[
f(x_1,\ldots,x_d)=\sum f^{(m_1,\ldots,m_d)}(x_1,\ldots,x_d)\in K\langle x_1,\ldots,x_d\rangle
\]
is a polynomial identity for the algebra $R$, then the multihomogeneous components $f^{(m_1,\ldots,m_d)}(x_1,\ldots,x_d)$ of $f(x_1,\ldots,x_d)$ are also
polynomial identities. Every multihomogeneous polynomial identity of total degree $m$ is equivalent to its linearization in $P_n$.
The same properties hold also for the central polynomials: their multihomogeneous components are either central polynomials or polynomial identities
and the linearization of a multihomogeneous central polynomial is also central.

One of the most powerful methods in the study of PI-algebras is based on representation theory of the symmetric group $S_m$ and the general linear group $GL_d(K)$.
For a background on representation theory of $S_m$ and $GL_d(K)$ see e.g. the books by James and Kerber \cite{JamKer}, Macdonald \cite{Mac} and Weyl \cite{We}.
The information for the applications to the theory of PI-algebras can be found in most of the books cited in the introduction, e.g. in \cite[Chapter 12]{Dr3}.

The symmetric group $S_m$ acts from the left on $P_m$ by
\[
\sigma:f(x_1,\ldots,x_m)\to f(x_{\sigma(1)},\ldots,x_{\sigma(m)}),\sigma\in S_m,f(x_1,\ldots,x_m)\in P_m,
\]
and as a left $S_m$-module $P_m$ is isomorphic to the group algebra $KS_m$. Hence $P_m$ decomposes as a direct sum of irreducible components $M(\lambda)$
\[
P_n=\bigoplus_{\lambda\vdash m}d_{\lambda}M(\lambda),
\]
where the summation runs on all partitions $\lambda=(\lambda_1,\ldots,\lambda_m)$ of $m$ and $d_{\lambda}=\dim M(\lambda)$.
For a fixed $\lambda$ every direct summand $M_i(\lambda)\cong M(\lambda)$, $i=1,\ldots,d_{\lambda}$, has an element $f^{(i)}_{\lambda}(x_1,\ldots,x_m)$ such that
every $S_m$-submodule $M(\lambda)$ of $P_m$ is generated by a sum
\[
f_{\lambda}(x_1,\ldots,x_m)=\sum_{i=1}^{d_{\lambda}}\xi_if^{(i)}_{\lambda}(x_1,\ldots,x_m),\xi_i\in K.
\]
For a PI-algebra $R$ the vector space $P_m\cap T(R)$ is an $S_m$-submodule of $P_m$ and it is an important problem to describe its structure as an $S_m$-module.
By the theorem of Regev \cite{Reg} the dimension of the factor space $P_m(R)=P_m/(P_m\cap T(R))$ is bounded by $a^m$ for some positive $a$ and asymptotically this is much smaller than
$\dim P_m=m!$. Hence it is more convenient to consider the factor module $P_m/(P_m\cap T(R))$ and the $S_m$-cocharacter sequence
\[
\chi_m(R)=\chi_{S_m}(P_m/(P_m\cap T(R)))=\sum_{\lambda\vdash m}m_{\lambda}\chi_{\lambda},
\]
where $\chi_{\lambda}$ is the $S_m$-character corresponding to the irreducible $S_m$-module $M(\lambda)$.

The general linear group $GL_d(K)$ acts canonically on the $d$-dimensional vector space $KX_d$ with basis $X_d=\{x_1,\ldots,x_d\}$
and this action is extended diagonally on $K\langle X_d\rangle$.
The homogeneous component $K\langle X_d\rangle^{(m)}$ decomposes as a $GL_d(K)$-module
\[
K\langle X_d\rangle^{(m)}=\bigoplus_{\lambda\vdash m}d_{\lambda}W_d(\lambda),
\]
where $W_d(\lambda)$ is the irreducible $GL_d(K)$-module corresponding to the partition $\lambda=(\lambda_1,\ldots,\lambda_d)\vdash m$,
with the same multiplicity $d_{\lambda}$ as the multiplicity of $M(\lambda)$ in the decomposition of the $S_d$-module $P_m$.
For a fixed $\lambda$ every direct summand $W^{(i)}(\lambda)\cong W(\lambda)$, $i=1,\ldots,d_{\lambda}$, has an element $w^{(i)}_{\lambda}(x_1,\ldots,x_d)$
which is multihomogeneous of degree $\lambda=(\lambda_1,\ldots,\lambda_d)$ such that
every $GL_d(K)$-submodule $W(\lambda)$ of $F_d(R)$ is generated by a sum
\[
w_{\lambda}(x_1,\ldots,x_d)=\sum_{i=1}^{d_{\lambda}}\xi_iw^{(i)}_{\lambda}(x_1,\ldots,x_d),\xi_i\in K.
\]
The polynomial $w_{\lambda}(x_1,\ldots,x_d)$ is called the highest weight vector of $W(\lambda)$ and is characterized by the property
\[
w_{\lambda}(x_1,\ldots,x_{q-1},x_q+x_p,x_{q+1},\ldots,x_d)=w_{\lambda}(x_1,\ldots,x_d)
\]
for all $1\leq p<q\leq d$.

The tensor product $M(\lambda)\otimes_K M(\mu)$ of an $S_{m_1}$-module $M(\lambda)$ and an $S_{m_2}$-module $M(\mu)$
has the structure of $S_{m_1}\times S_{m_2}$-module.
The decomposition into a sum of irreducible $S_{m_1+m_2}$-modules of the induced $S_{m_1+m_2}$-module $(M(\lambda)\otimes_K M(\mu))\uparrow S_{m_1+m_2}$
can be obtained by the combinatorial Littlewood-Richardson rule.
The same rule gives the decomposition into a sum of irreducible $GL_d(K)$-submodules
of the tensor product $W(\lambda)\otimes_K W(\mu)$ of two irreducible polynomial
$GL_d(K)$-modules $W(\lambda)$ and $W(\mu)$.
In the special case of $d=2$ it has the following simple form.
If $\lambda=(p_1+p_2,p_2)$, $\mu=(q_1+q_2,q_2)$ and $p_1\geq q_1$, then
\[
W(\lambda)\otimes_KW(\mu)\cong W(p_1+p_2+q_1+q_2,p_2+q_2)\oplus
\]
\[
\oplus W(p_1+p_2+q_1+q_2-1,p_2+q_2+1)\oplus\cdots\oplus W(p_1+p_2+q_2,p_2+q_2+q_1).
\]
In the sequel all commutators are left normed:
\[
u_m=[x_{i_1},x_{i_2},x_{i_3},\ldots,x_{i_{m-1}},x_{i_m}]=[[x_{i_1},x_{i_2},x_{i_3},\ldots,x_{i_{m-1}}],x_{i_m}]].
\]
With few exceptions in Section \ref{section 4} we also assume that the commutators satisfy the condition $i_1>i_2\leq i_3\leq\cdots\leq i_{m-1}\leq i_m$.
A special role in the study of polynomial identities of unitary algebras is played by the proper polynomials which are linear combinations of products of commutators.
Modulo the ideal of $K\langle X_d\rangle$ generated by products of two commutators, the vector space spanned by commutators of length $m$ is isomorphic to the $GL_d(K)$-module $W(d-1,1)$.
Similarly, the vector space spanned by products of $k$ commutators of length $m_1,\ldots,m_k$ modulo the ideal generated by products of $k+1$ commutators
is isomorphic to the tensor product of $GL_d(K)$-modules
\[
W(m_1-1,1)\otimes_K\cdots\otimes_KW(m_k-1,1).
\]

The ideal $K\langle X_d\rangle\cap T(R)$ is a $GL_d(K)$-submodule of $K\langle X_d\rangle$ and this gives the structure of a $GL_d(K)$-module
of the relatively free algebra $F_d(R)=K\langle X_d\rangle/(K\langle X_d\rangle\cap T(R))$ in the variety of algebras $\text{var}(R)$ generated by $R$.
By a theorem of Berele \cite{Ber} and Drensky \cite{Dr1} the homogeneous component $F_d(R)^{(m)}$ of degree $m$ of $F_d(R)$ decomposes as
\[
F_d(R)^{(m)}=\bigoplus_{\lambda\vdash m}m_{\lambda}W_d(\lambda),
\]
and for $\lambda=(\lambda_1,\ldots,\lambda_d)\vdash n$ the multiplicity $m_{\lambda}$ is the same as in the $m$-th cocharacter $\chi_m(R)$.

\section{Possible approaches to the problem}
In this section we shall describe several approaches to the problem for the existence of central polynomials of degree lower than the degree in the conjecture of Formanek.

\subsection{The method of multilinear polynomials}
We can use the following method to determine whether the matrix algebra $M_n(K)$ has central polynomials of degree $m$.
It is sufficient to consider multilinear polynomials only. Let
\[
f(x_1,\ldots,x_m)=\sum_{\sigma\in S_m}\xi_{\sigma}x_{\sigma(1)}\cdots x_{\sigma(m)}
\]
be the candidate for a central polynomial of degree $m$ with $m!$ unknown coefficients $\xi_{\sigma}\in K$, $\sigma\in S_m$.
We replace the variables $x_i$ with all possible matrices $E_{p_iq_i}$, $p_i,q_i=1,\ldots,n$, where $\{E_{ab}\mid 1\leq a,b\leq n\}$ is the standard basis of $M_n(K)$:
\[
f(E_{p_1q_1},\ldots,E_{p_mq_m})=\sum_{r,s=1}^nf^{(p,q)}_{rs}E_{rs},
\]
where $f^{(p,q)}_{rs}$, $(p,q)=((p_1,q_1),\ldots,(p_m,q_m))$, are linear combinations of $\xi_{\sigma}$.
We consider the linear homogeneous system with $m!$ unknowns $\xi_{\sigma}$
\[
f^{(p,q)}_{rs}=0,\,r,s=1,\ldots,m,\,(p,q)=((p_1,q_1),\ldots,(p_m,q_m)).
\]
The solutions of the system are all multilinear polynomial identities of degree $m$ for $M_n(K)$. If we determine the rank of the system $\text{rank}_{\text{PI}}$, this would give us
the dimension of the vector space of the polynomial identities. Then consider the system which consists of the equations with $r\not=s$ and add the equations of the form

\[
f^{(p,q)}_{11}=\cdots=f^{(p,q)}_{nn}.
\]
The solutions of the system give the polynomial identities and the central polynomials of degree $m$. Let the rank of the system be $\text{rank}_{\text{central}}$.
Clearly $\text{rank}_{\text{PI}}\leq\text{rank}_{\text{central}}$.
If the ranks of both systems are equal, this means that there are no central polynomials. If $\text{rank}_{\text{PI}}<\text{rank}_{\text{central}}$ using computational methods from linear algebra,
e.g. the Gauss elimination method, we can find $\text{rank}_{\text{central}}-\text{rank}_{\text{PI}}$ solutions of the second system which are linearly independent modulo the solutions of the first system
and these polynomials will span the vector space of the central polynomials.

If we show that there are no central polynomials of degree $m$, this immediately implies that there also are no central polynomials of degree $j<m$:
If $c(x_1,x_2,\ldots,x_j)$ is a central polynomial of degree $j<m$ and $c(r_1,\ldots,r_j)\not=0$ for some $r_1,\ldots,r_j\in M_n(K)$,
then the polynomial $c(x_1x_{j+1}\cdots x_m,x_2,\ldots,x_j)$ is either central or a polynomial identity because it takes only central values evaluated on $M_n(K)$.
Since
\[
c(r_1I_n\cdots I_n,r_2,\ldots,r_j)=c(r_1,\ldots,r_j)\not=0,
\]
we conclude that $c(x_1x_{j+1}\cdots x_m,x_2,\ldots,x_j)$ is a central polynomial of degree $m$.

This method is not very effective even for small $n$ and $m$. For example, if we want to prove that $M_4(K)$ does not have central polynomials of degree 12,
we have to find the ranks of two homogeneous linear systems with $12!=479001600$ unknown coefficients $\xi_{\sigma}\in K$, $\sigma\in S_{12}$.

\subsection{The method of $S_m$-representation theory}
Since the vector space $P_m$ of the multilinear polynomials of degree $m$ in $K\langle X\rangle$ is a left $S_m$-module,
for each partition $\lambda=(\lambda_1,\ldots,\lambda_m)$ of $m$ we consider a linear homogeneous system with
number of unknowns equal to the degree $d_{\lambda}$ of the associated irreducible character $\chi_{\lambda}$.
Instead of linear combinations of monomials $x_{\sigma(1)}\cdots x_{\sigma(m)}$ we consider the polynomials in $P_m$
\[
\sum_{i=1}^{d_{\lambda}}\eta_if_{\lambda}^{(i)}(x_1,\ldots,x_m)
\]
where every linear combinations of the polynomials $f_{\lambda}^{(i)}$ generates an  irreducible $S_m$-module.
Then we repeat the method described in the previous subsection for the linear combinations of monomials and compare the ranks of the homogeneous linear systems.
The number of the unknowns is much smaller than $m!$ (for example, for $m=12$ the number of the unknowns is $\leq 7700$), but the polynomials $f_{\lambda}^{(i)}$ are more complicated.
This method was used for example by Bondari \cite{Bo}.

The method of representations of $S_m$ has been used in many other cases.
For example, Leron \cite{Ler} found that for $n>2$ all polynomial identities of $M_n(K)$ of degree $2n+1$
follow from the standard identity $s_{2n}(x_1,\ldots,x_{2n})=0$.
Olsson and Regev \cite{OlsReg} calculated the $S_{m+1}$-character of degree $m+1$
of the consequences of degree $m+1$ of the standard identity $s_m(x_1,\ldots,x_m)=0$.

\begin{theorem}\cite[Theorem 20]{Ler}, \cite[Theorem 1]{OlsReg}\label{Olsson and Regev}
The $S_{2n+1}$-character of the
multilinear polynomial identities of degree $2n+1$ of $M_n(K)$, $n\geq 3$, is
\[
\chi_{S_{2n+1}}(P_{2n+1}\cap T(M_n(K)))=\chi(3,1^{2n-2})+\chi(2^2,1^{2n-1})+
\]
\[
+3\chi(2,1^{2n-1})+\chi(1^{2n+1}).
\]
\end{theorem}

For applications to the study of polynomial identities of nonassociative algebras see e.g.
the survey by Bremner, Madariaga and Peresi \cite{BreMaPe}.

\subsection{The method of multihomogeneous polynomials}
As in the case of multilinear polynomials, consider the polynomial $f(x_1,\ldots,x_d)$ which is homogeneous of multidegree $\mu=(\mu_1,\ldots,\mu_d)$
\[
f(x_1,\ldots,\mu_d)=\sum_i\vartheta_ix_{i_1}\cdots x_{i_m}.
\]
The number of unknowns $\vartheta_i$ is equal to the number of monomials of multidegree $\mu$ in $K\langle X_d\rangle$ which is
$\displaystyle \frac{(\mu_1+\cdots+\mu_d)!}{m_1!\cdots m_d!}$ but we have to replace $x_1,\ldots,x_d$ with generic $n\times n$ matrices
\[
y_1=\left(\begin{matrix}y^{(1)}_{11}&\cdots&y^{(1)}_{1n}\\
\vdots&\ddots&\vdots\\
y^{(1)}_{n1}&\cdots&y^{(1)}_{nn}\\
\end{matrix}\right),\ldots,
y_d=\left(\begin{matrix}y^{(d)}_{11}&\cdots&y^{(d)}_{1n}\\
\vdots&\ddots&\vdots\\
y^{(d)}_{n1}&\cdots&y^{(d)}_{nn}\\
\end{matrix}\right).
\]
To handle the central polynomials in two variables of degree 12 we have to consider 6 linear systems
with 12, 66, 220, 495, 792 and 924 unknowns, for the cases of bigraded identities of degree
(11,1), (10,2), (9,3), (8,4), (7,5) and (6,6), respectively.

\subsection{The method of representations of $GL_d(K)$}\label{the method of GL-representations}
In the theory of PI-algebras the results obtained in the language of representation theory of the general linear group
can be easily restated in the language of representation theory of the symmetric group.
The advantage is that the considered polynomials are simpler than in the case of the symmetric group.
For applications of the method to PI-algebras see e.g. \cite[Chapter 12]{Dr3}.

For our purposes we can fix a partition $\lambda=(\lambda_1,\ldots,\lambda_d)$ and consider
the highest weight vector of the irreducible $GL_d(K)$-module $W(\lambda)$
\[
w_{\lambda}(x_1,\ldots,x_d)=\sum_{i=1}^{d_{\lambda}}\xi_iw^{(i)}_{\lambda}(x_1,\ldots,x_d),\xi_i\in K.
\]
As in the previous subsection we replace the variables $x_1,\ldots,x_d$ with generic $n\times n$ matrices
and compare the ranks of the homogeneous linear systems corresponding to the polynomial identities and the central polynomials.
As in the case when we use representations of $S_m$, the number of unknowns $\xi_i$ is much smaller than
in the case when we work with multihomogeneous polynomials.
For example, if we are interested in central polynomials in two variables of degree 12,
we have to solve 6 systems with, respectively, 11, 54, 154, 275, 297 and 132 unknowns.

The representation theory of $GL_d(K)$ was used by Benanti and Drensky \cite{BenDr} who found all consequences of degree $m+2$
of the standard identity $s_m(x_1,\ldots,x_m)=0$.

\begin{theorem}\cite[Theorem 3.1]{BenDr}\label{Consequences of S_m}
The $S_{m+2}$-character of the
multilinear consequences of degree $m+2$ of
the standard polynomial $s_m(x_1,\ldots,x_m)$, $m\geq 5$, is
\[
\chi_{S_{m+2}}(P_{m+2}\cap\langle s_m\rangle^T)=
2\chi(4,2,1^{m-4})+4\chi(4,1^{m-2})+2\chi(3,2^2,1^{m-5})+
\]
\[
+8\chi(3,2,1^{m-3})+9\chi(3,1^{m-1})+4\chi(2^3,1^{m-4})+
\]
\[
+8\chi(2^2,1^{m-2})+5\chi(2,1^m)+\chi(1^{m+2}).
\]
\end{theorem}

There was a conjecture that for $n>2$ all polynomial identities for $M_n(K)$ of degree $2n+2$
follow from the standard identity $s_{2n}=0$.
For $n=3$ this was confirmed in \cite{DrKas1}.
The computer calculations for the polynomial identities for $M_n(K)$ of degree $2n+2$ for $n=4,5$
were performed in \cite{BenDeDrKo1} using $64$ processors on the Cray T3E at the Lawrence Berkeley Lab
(see \cite{BenDeDrKo2} for details).
It took a total of about $8$ hours to complete
the computations for the identities of degree 12
for $5\times 5$ matrices from start to finish.
(In order to compare the results with the known facts,
six hours of this time were spent for calculations which follow from
other arguments.) As a result, the conjecture for the polynomial identities of degree $2n+2$ for $M_n(K)$ was confirmed also for $n=4,5$.

\begin{theorem}\label{PI of degree n+2}
All polynomial identities of degree $2n+2$ for $M_n(K)$, $n=3,4,5$, follow from the standard identity $s_{2n}(x_1,\ldots,x_{2n})=0$.
\end{theorem}

The calculations for $n=5$ and $m=12$ as performed in
\cite{BenDeDrKo1} demonstrate the possibilities and the advantages of the
computational method based on representations of $GL_d(K)$. As in the case of representations of $S_{12}$
we have to consider 77 systems with up to 7700 unknowns but the computations are much simpler.
The methods used in \cite{BenDeDrKo1} allowed only to find sufficiently good bounds for the ranks of the systems.
But these bounds meet the conjectured ones.
One would decrease the number of the systems to 33, the number of unknowns
to 5775 and the computing time to two hours if one used all the known information for the polynomial identities of $M_n(K)$.

\subsection{Representations of $GL_d(K)$ combined with PI-theory}
The computations for the central polynomials of minimal degree can be simplified significantly if we use more information from the theory of PI-algebras.

\begin{lemma}\label{relation with S_m-characters}
Let $M_n(K)$ have a central polynomial $c(x_1,\ldots,x_d)$ of degree $m$. Then the multilinear component $P_{m+1}\cap T(M_n(K))$
contains an irreducible $S_{m+1}$-submodule $M(\lambda)$ for some $\lambda=(\lambda_1,\ldots,\lambda_d)\vdash m+1$.
\end{lemma}

\begin{proof}
The algebra $M_n(K)$, $n>1$, does not have central polynomials in one variable.
If $c(x_1,\ldots,x_d)$ is a central polynomial of degree $m$, then the commutator $[c(x_1,\ldots,x_d),x_1]$ is a nontrivial polynomial identity in $d$ variables.
The linearization of $[c(x_1,\ldots,x_d),x_1]$ generates an $S_{m+1}$-submodule of $P_{m+1}$ which is a direct sum of some $M(\lambda_1,\ldots,\lambda_d)$.
\end{proof}

\begin{corollary}\label{central polynomials of degree <10}
The matrix algebra $M_4(K)$ does not have central polynomials in two variables of degree $\leq 9$
and polynomial identities in two variables of degree $\leq 10$.
\end{corollary}

\begin{proof}
By the Amitsur-Levitzky theorem the standard polynomial $s_8$ is the only polynomial identity for $M_4(K)$ of degree $\leq 8$.
It generates the $S_8$-module $M(1^8)$. By Theorems \ref{Olsson and Regev}, \ref{Consequences of S_m} and \ref{PI of degree n+2}
the $S_m$-module of the polynomial identities for $M_4(K)$ of degree $m=9,10$ are
\[
\chi_{S_9}(P_{9}\cap T(M_4(K)))=\chi(3,1^6)+\chi(2^2,1^5)+3\chi(2,1^7)+\chi(1^9),
\]
\[
\chi_{S_{10}}(P_{m+2}\cap T(M_4(K)))=
2\chi(4,2,1^4)+4\chi(4,1^6)+2\chi(3,2^2,1^3)+
\]
\[
+8\chi(3,2,1^5)+9\chi(3,1^7)+4\chi(2^3,1^4)
+8\chi(2^2,1^6)+5\chi(2,1^8)+\chi(1^{10}).
\]
This gives the proof for the polynomial identities in two variables of degree $\leq 10$.
Since $P_m\cap T(M_4(K))$ does not contain $S_m$-submodules $M(\lambda_1,\lambda_2)$ for $m\leq 10$, Lemma \ref{relation with S_m-characters} gives that
$M_4(K)$ does not have central polynomials in two variables of degree $\leq 9$.
\end{proof}

Let $UT_n(K)$ be the algebra of $n\times n$ upper triangular matrices.
Yu. Maltsev \cite{Mal} proved that the T-ideal $T(UT_n(K))$ is generated by the polynomial identity
\[
[x_1,x_2]\cdots[x_{2n-1},x_{2n}]=0.
\]
This result allows to reduce the search for central polynomials to polynomials of a special form.

\begin{lemma}\cite[Proposition 4.2 (i)]{DrKas1}\label{reduction to proper}
If $c(x_1,\ldots,x_d)$ is a central polynomial for $M_n(K)$ of minimal degree in $K\langle X_d\rangle$, then
$c(x_1,\ldots,x_d)$ is a proper polynomial in $T(UT_n(K))$.
\end{lemma}

\begin{proof}
It is well known, see e.g. \cite[Proposition 4.3.3]{Dr3} that every polynomial $c(x_1,\ldots,x_d)$ in $K\langle X_d\rangle$ can be written in the form
\[
c(x_1,\ldots,x_d)=\sum x_1^{a_1}\cdots x_d^{a_d}c_a(x_1,\ldots,x_d),
\]
where $c_a(x_1,\ldots,x_d)$ is a linear combination of products of commutators, i.e. proper polynomials. Since we are searching central polynomials, we may assume that
all $c_a(x_1,\ldots,x_d)$ are not polynomial identities for $M_n(K)$. The polynomial $c(1+x_1,\ldots,1+x_d)$ is also central for $M_n(K)$ and
\[
c(1+x_1,\ldots,1+x_d)=\sum (1+x_1)^{a_1}\cdots (1+x_d)^{a_d}c_a(x_1,\ldots,x_d).
\]
The multihomogeneous components of $c(1+x_1,\ldots,1+x_d)$ are also central polynomials or polynomial identities for $M_n(K)$.
Hence this holds also for the proper components $c_a(x_1,\ldots,x_d)$ of minimal degree. Since they are not polynomial identities, they are central polynomials.

We can write $c_a(x_1,\ldots,x_d)$ as a linear combination of products of left normed commutators
$u=[x_{i_1},x_{i_2},\ldots,x_{i_k}]$, $i_1>i_2\leq\cdots\leq i_k$:
\[
c_a(x_1,\ldots,x_d)=\sum_{m\geq 1}c_a^{(m)},\text{ where }
c_a^{(m)}=\sum\alpha^{(m)}_ju^{(m)}_{j_1}\cdots u^{(m)}_{j_m}
\]
is the component with $m$ commutators in the products.The sum
\[
c^{(m)}(x_1,\ldots,x_d)=\sum_{m=1}^{n-1}c_a^{(m)}(x_1,\ldots,x_d)
\]
is not a polynomial identity for the algebra $UT_n(K)$. Hence there exist $n\times n$ upper triangular matrices $r_1,\ldots,r_d$ such that
$c^{(m)}(r_1,\ldots,r_d)\not=0$ in $UT_n(K)\subset M_n(K)$. Since $c^{(m)}(x_1,\ldots,x_d)$ is a central polynomial for $M_n(K)$, we obtain that
$c^{(m)}(r_1,\ldots,r_d)$ is a nonzero scalar matrix. But this is impossible because the commutator of two upper triangular matrices is with zero diagonal
and the same holds for $c^{(m)}(r_1,\ldots,r_d)$.
Hence each $c_a(x_1,\ldots,x_d)$ is a linear combination of products of $\geq n$ commutators and $c_a(x_1,\ldots,x_d)\in T(UT_n(K))$.
\end{proof}

\begin{remark}\label{reduction to proper PI}
For unitary algebras all polynomial identities follow from the proper ones. Since
$T(M_n(K))\subset T(UT_n(K))$, this immediately implies that if
$f(x_1,\ldots,x_d)$ is a polynomial identity for $M_n(K)$ of minimal degree in $K\langle X_d\rangle$, then
$f(x_1,\ldots,x_d)$ is a proper polynomial in $T(UT_n(K))$.
\end{remark}

Following Bergman \cite{Berg} we define a linear map
\[
\alpha:K[t_0,t_1,\ldots,t_n]\times K\langle x_1,\ldots,x_n\rangle^{(n)}\to K\langle x_1,\ldots,x_n,y\rangle
\]
by
\[
\alpha(t_0^{a_0}t_1^{a_1}\cdots t_{n-1}^{a_{n-1}}t_n^{a_n},x_{i_1}\cdots x_{i_n})=y^{a_0}x_{i_1}y^{a_1}\cdots y^{a_{n-1}}x_{i_n}y^{a_n}.
\]
Then every polynomial in $K\langle x_1,\ldots,x_n,y\rangle$ which is homogeneous of degree $n$ in $x_1,\ldots,x_n$ can be written in the form
\[
f(x_1,\ldots,x_n,y)=\sum_i\alpha(h_i(t_0,t_1,\ldots,t_n),x_{i_1}\cdots x_{i_n}),
\]
$h_i(t_0,t_1,\ldots,t_n)\in K[t_0,t_1,\ldots,t_n]$.
The following theorem of Bergman \cite{Berg} describes the polynomial identities of $M_n(K)$ of this form which are multilinear in $x_1,\ldots,x_n$
and the central polynomials $c(x_1,\ldots,x_n,y)$ of $M_n(K)$ which are of degree $n$ in $x_1,\ldots,x_n$.

\begin{theorem}\cite[Part II]{Berg}\label{central polynomials of degree (n,k)}
{\rm (i)} The polynomial
\[
f(x_1,\ldots,x_n,y)=\sum_{\sigma\in S_n}\alpha(h_{\sigma}(t_0,t_1,\ldots,t_n),x_{\sigma(1)}\cdots x_{\sigma(n)})
\]
is a polynomial identity for $M_n(K)$ if and only if $h_{\sigma}(t_0,t_1,\ldots,t_n)$ is divisible by all $t_p-t_q$, $0\leq p<q\leq n$, for all $\sigma\in S_n$.

{\rm(ii)} The central polynomials $c(x_1,\ldots,x_n,y)$ which are multilinear in $x_1,\ldots,x_n$ are linear combinations of
\[
\sum_{i=1}^n\alpha(h_{\sigma}(t_0,t_1,\ldots,t_n),x_{\sigma(i)},\ldots,x_{\sigma(n)},x_{\sigma(1)}\cdots x_{\sigma(i-1)}),
\]
where $h_{\sigma}(t_0,t_1,\ldots,t_n)$, $\sigma\in S_n$, is divisible by all
\[
t_0-t_p, t_p-t_n,1\leq p\leq n-1,\,(t_p-t_q)^2, 1\leq p<\leq q\leq n-1,
\]
The polynomial $h(t_0,t_1,\ldots,t_n)$ for the central polynomial
\[
c(x,y)=\alpha(h(t_0,t_1,\ldots,t_n),x,\ldots,x)
\]
satisfies the same divisibility conditions.
In particular, $M_n(K)$ does not have central polynomials $c(x_1,\ldots,x_n,y)$ of degree $n$ in $x_1,\ldots,x_n$ and of degree $m<n(n-1)$ in $y$.
\end{theorem}

\section{Central polynomials in two variables for $M_4(K)$}\label{section 4}

\subsection{Preliminaries}\label{preliminaries for two variables}
To show that $M_4(K)$ does not have central polynomials in two variables of degree $\leq 12$ we shall use the following:
\begin{itemize}
\item[(i)] By Corollary \ref{central polynomials of degree <10} it is sufficient to consider polynomials of degree 10, 11 and 12.
\item[(ii)] By Lemma \ref{reduction to proper} it is sufficient to search central polynomials which are highest weight vectors
$w_{\lambda}(x,y)$, $\lambda=(\lambda_1,\lambda_2)$, which can be expressed as linear combinations of products of $\geq 4$ commutators.
\item[(iii)] By Theorem \ref{central polynomials of degree (n,k)} (ii) it is sufficient to consider only the cases $\lambda=(\lambda_1,\lambda_2)$ with $\lambda_2\geq 5$.
\end{itemize}

Concerning polynomial identities in two variables of degree $\leq 12$, we have the following arguments. By Theorem \ref{central polynomials of degree (n,k)} (i)
the polynomial identities for $M_n(K)$ in a $GL_d(K)$-submodule $W(\lambda_1,\lambda_2,\ldots,\lambda_d)$ of $K\langle X_d\rangle$ for $\lambda_2+\cdots+\lambda_d=n$
are of degree $\displaystyle \geq n+\binom{n+1}{2}=\frac{n(n+3)}{2}$. For $n=4$ this implies that for $(\lambda_1,4)$ the algebra $M_4(K)$ does not have polynomial identities of degree $<14$.
Theorems \ref{Olsson and Regev}, \ref{Consequences of S_m} and \ref{PI of degree n+2} handle the case of polynomial identities of degree $\leq 10$.
Hence, as in the case of central polynomials it is sufficient to study the cases $\lambda=(6,5),(7,5)$ and (6,6).

Let $\text{\rm Com}_m$ be the vector subspace of $K\langle x,y\rangle$ with basis all proper commutators $[y,x,\ldots,x,y,\ldots,y]$ of length $m$.
As we stated in the Preliminaries, the vector space spanned by products $\text{\rm Com}_{m_1}\cdots \text{\rm Com}_{m_k}$
of $k$ commutators of length $m_1,\ldots,m_k$ modulo the ideal generated by products of $k+1$ commutators
is isomorphic to the tensor product of $GL_d(K)$-modules
\[
W(m_1-1,1)\otimes_K\cdots\otimes_KW(m_k-1,1).
\]
The tensor product $W=W(p_1+p_2,p_2)\otimes_KW(m+1,1)$ decomposes as
\[
W=W(p_1+p_2+m+1,p_2+1)\oplus\cdots\oplus W(p_1+p_2+1,p_2+m+1)
\]
if $p_1\geq m$ and as
\[
W=W(p_1+p_2+m+1,p_2+1)\oplus\cdots\oplus W(p_2+m+1,p_2+p_1+1)
\]
if $p_1<m$.

We shall need the following decompositions obtained by direct computations:

\noindent (i) For products of 4 commutators of degree 10:
\[
W(3,1)\otimes_KW(1,1)^{\otimes 3}=W(6,4),
\]
\[
W(2,1)^{\otimes 2}\otimes_KW(1,1)^{\otimes 2}=W(6,4)\oplus W(5,5);
\]
(ii) For products of 5 commutators of degree 10:
\[
W(1,1)^{\otimes 5}=W(5,5);
\]
(iii) For products of 4 commutators of degree 11:
\[
W(4,1)\otimes_KW(1,1)^{\otimes 3}=W(7,4),
\]
\[
W(3,1)\otimes_KW(2,1)\otimes_KW(1,1)^{\otimes 2}=W(7,4)\oplus W(6,5),
\]
\[
W(2,1)^{\otimes 3}\otimes_KW(1,1)=W(7,4)\oplus 2W(6,5);
\]
(iv) For products of 5 commutators of degree 11:
\[
W(2,1)\otimes_KW(1,1)^{\otimes 4}=W(6,5);
\]
(v) For products of 4 commutators of degree 12:
\[
W(5,1)\otimes_K W(1,1)^{\otimes 3}=W(8,4),
\]
\[
W(4,1)\otimes_KW(2,1)\otimes_KW(1,1)^{\otimes 2}=W(8,4)\oplus W(7,5),
\]
\[
W(3,1)^{\otimes 2}\otimes_KW(1,1)^{\otimes 2}=W(8,4)\oplus W(7,5)\oplus W(6,6),
\]
\[
W(3,1)\otimes_KW(2,1)^{\otimes 2}\otimes_KW(1,1)=W(8,4)\oplus 2W(7,5)\oplus W(6,6),
\]
\[
W(2,1)^{\otimes 4}=W(8,4)\oplus 3W(7,5)\oplus 2W(6,6);
\]
(vi) For products of 5 commutators of degree 12:
\[
W(3,1)\otimes_K W(1,1)^{\otimes 4}=W(7,5),
\]
\[
W(2,1)^{\otimes 2}\otimes_K W(1,1)^{\otimes 3}=W(7,5)\oplus W(6,6);
\]
(vii) For products of 6 commutators of degree 12:
\[
W(1,1)^{\otimes 6}=W(6,6).
\]

\subsection{Computations for degree 10}
In all the computations for degree 10, 11 and 12 shall need the highest weight vectors of
\[
W(\lambda_1,\lambda_2)\subset \text{\rm Com}_{m_1}\cdots \text{\rm Com}_{m_k}, \lambda_2\geq 5, k\geq 4.
\]

\begin{lemma}\label{HWV of degree 10}
The following polynomials are highest weight vectors for $\lambda=(5,5)$:
\[
w_1(x,y)=([y,x,x][y,x,y]-[y,x,y][y,x,x])[y,x]^2\in \text{\rm Com}_3^2\text{\rm Com}_2^2,
\]
\[
w_2(x,y)=[y,x]^5\in \text{\rm Com}_2^5.
\]
\end{lemma}

\begin{proof}
Both polynomials are bihomogeneous of degree $(5,5)$. They are highest weight vectors if and only if
$w_i(x,y+x)=w_i(x,y)$, $i=1,2$, which can be verified by direct computations.
\end{proof}

\begin{proposition}\label{no central polynomials of degree 10}
The algebra $M_4(K)$ does not have central polynomials $c(x,y)$ of degree $10$.
\end{proposition}

\begin{proof}
The only case of $\lambda=(\lambda_1,\lambda_2)\vdash 10$ with $\lambda_2\geq 5$ is $\lambda=(5,5)$.
By the computations in Subsection \ref{preliminaries for two variables}
the product of $k\geq 4$ commutators contains a $GL_2(K)$-submodule $W(5,5)$ if:

(i) The product $\text{\rm Com}_{m_1}\cdots\text{\rm Com}_{m_4}$ consists of two copies of $\text{\rm Com}_3$ and two copies of $\text{\rm Com}_2$.
The possible cases for $(m_1,m_2,m_3,m_4)$ are 6:
\[
(3,3,2,2), (3,2,3,2), (3,2,2,3), (2,3,3,2), (2,3,2,3), (2,2,3,3).
\]

(ii) The product consists of five commutators of length 2.

In the case (i) the highest weight vectors $w^{(1)},\ldots,w^{(6)}$ are similar to the first polynomial $w_1(x,y)$ in Lemma \ref{HWV of degree 10}:
\[
w^{(i)}(x,y)=[y,x]^p([y,x,x][y,x]^q[y,x,y]-[y,x,y][y,x]^q[y,x,x])[y,x]^r,
\]
$(p,q,r)=(2,0,0),(1,1,0),(1,0,1),(0,2,0),(0,1,1),(0,0,2)$.
In the case (ii) we have only the possibility $w^{(7)}(x,y)=w_2(x,y)=[y,x]^5$.
The possible central polynomials are the highest weight vectors
\[
w_{(5,5)}(x,y)=\sum_{i=1}^7\xi_iw^{(i)}(x,y)
\]
with unknown coefficients $\xi_i\in K$, $i=1,\ldots,7$. We can replace $x$ and $y$ with generic $4\times 4$ matrices
and apply the method described in Subsection \ref{the method of GL-representations}. There is a general principle
that in the generic matrix algebra we may assume that the first generic matrix is diagonal. Hence we may replace $x$ by the diagonal generic matrix
\[
u=\left(\begin{matrix}u_1&0&0&0\\
0&u_2&0&0\\
0&0&u_3&0\\
0&0&0&u_4\\
\end{matrix}\right).
\]
Instead of replacing $y$ by another generic matrix, we have done the following which has simplified the computations.
We have replaced $y$ by the matrix
\[
v=\left(\begin{matrix}0&1&0&0\\
0&0&1&0\\
0&0&0&1\\
1&0&0&0\\
\end{matrix}\right).
\]
If $w_{(5,5)}(x,y)$ is a central polynomial, then in
\[
w_{(5,5)}(u,v)=\sum_{a,b=1}^4z_{ab}E_{ab}
\]
all coefficients $z_{ab}$ are equal to 0 for $a\not=b$. We have computed
\[
z_{12}=\sum_if_i(\xi_1,\ldots,\xi_7)u_1^{i_1}u_2^{i_2}u_3^{i_3}u_4^{i_4}, i_1+i_2+i_3+i_4=5.
\]
The equality $z_{12}=0$ implies that all $f_i(\xi_1,\ldots,\xi_7)$ are also equal to 0.
In this way we obtain a homogeneous linear system
\[
f_i(\xi_1,\ldots,\xi_7)=0,i_1+i_2+i_3+i_4=5,
\]
with unknowns $\xi_1,\ldots,\xi_7$. It has turned out that the system has only the trivial solution which shows that
$M_4(K)$ does not have polynomial identities and central polynomials in $W(5,5)$.
\end{proof}

\subsection{Computations for degree 11}
The proof of the following lemma is similar to the proof of Lemma \ref{HWV of degree 10}.

\begin{lemma}\label{HWV of degree 11}
The following polynomials are highest weight vectors for $\lambda=(6,5)$:
\[
w_1(x,y)=([y,x,x][y,x,y]-[y,x,y][y,x,x])[y,x]^2\in \text{\rm Com}_4\text{\rm Com}_3\text{\rm Com}_2^2,
\]
\[
w_2(x,y)=([y,x,x][y,x,y]-[y,x,y][y,x,x])[y,x,x][y,x]\in \text{\rm Com}_3^3\text{\rm Com}_2
\]
(the skew symmetry between $x$ and $y$ in the end of the commutators of length $3$ is between the first and the second commutators of length $3$),
\[
w_3(x,y)=[y,x,x]([y,x,x][y,x,y]-[y,x,y][y,x,x])[y,x]\in\text{\rm Com}_3^3\text{\rm Com}_2
\]
(with skew symmetry between the second and the third commutators of length $3$),
\[
w_4(x,y)=[y,x,x][y,x]^4\in\text{\rm Com}_3\text{\rm Com}_2^4.
\]
\end{lemma}

\begin{proposition}\label{no central polynomials of degree 11}
The algebra $M_4(K)$ does not have central polynomials and polynomial identities of degree $11$ in $K\langle x,y\rangle$.
\end{proposition}

\begin{proof}
We follow the main steps of the proof of Proposition \ref{no central polynomials of degree 10}.
It is sufficient to consider the case $\lambda=(6,5)$.
There is one submodule $W(6,5)$ in the decomposition of $\text{\rm Com}_4\text{\rm Com}_3\text{\rm Com}_2^2$
into a sum of irreducible $GL_2(K)$-modules and there are 12 possibilities for the ordering of
$\text{\rm Com}_4$, $\text{\rm Com}_3$ and $2\text{\rm Com}_2$ as a product $\text{\rm Com}_{m_1}\cdots \text{\rm Com}_{m_4}$.
Hence there are 12 highest weight vectors similar to $w_1(x,y)$.
Depending of the position of $\text{\rm Com}_2$ in the product of $3\text{\rm Com}_3$ and $\text{\rm Com}_2$ there are 4 possibilities
for the highest weight vectors similar to $w_2(x,y)$ and $w_3(x,y)$.
Finally, there are 5 possibilities for polynomials similar to $w_4(x,y)$.
The central polynomials and polynomial identities which we search for are the highest weight vectors
\[
w_{(6,5)}(x,y)=\sum_{i=1}^{25}\xi_iw^{(i)}(x,y)
\]
with unknown coefficients $\xi_i\in K$, $i=1,\ldots,25$. As in Proposition \ref{no central polynomials of degree 10}
we replace $x$ by the diagonal generic matrix $u$ and $y$ by the matrices $v$ and
\[
v_1=\left(\begin{matrix}0&1&1&0\\
0&0&1&0\\
0&0&0&1\\
1&0&0&0\\
\end{matrix}\right).
\]
Expressing $w_{(6,5)}(u,v)$ and $w_{(6,5)}(u,v_1)$ in the form
\[
w_{(6,5)}(u,v)=w_{(5,5)}(u,v)=\sum_{a,b=1}^4z_{ab}E_{ab},
\]
\[
w_{(6,5)}(u,v_1)=w_{(5,5)}(u,v)=\sum_{a,b=1}^4z'_{ab}E_{ab},
\]
we consider the homogeneous linear system obtained from the conditions $z_{12}=z'_{13}=z'_{14}=0$.
Using the first condition $z_{12}=0$ only, we reduce the system to a system with 10 unknowns only. The second condition $z'_{13}=0$
eliminates 8 more unknowns. Finally, the condition $z'_{14}=0$ shows that the obtained system with two unknowns has only the trivial solution.
\end{proof}

\subsection{Computations for degree 12}
The highest weight vectors we need are described in the following lemma.
The proof is similar to the proof of Lemma \ref{HWV of degree 10}.

\begin{lemma}\label{HWV of degree 12}
{\rm (i)} The following polynomials are highest weight vectors in $\text{\rm Com}_{m_1}\cdots \text{\rm Com}_{m_k}$, $m=(m_1,\ldots,m_k)$, for $\lambda=(7,5)$:
\[
(5,3,2^2):w_1(x,y)=([y,x,x,x,x][y,x,y]-[y,x,x,x,y][y,x,x])[y,x]^2, 
\]
\[
(4^2,2^2):w_2(x,y)=([y,x,x,x][y,x,x,y]-[y,x,x,y][y,x,x,x])[y,x]^, 
\]
\[
(4,3^2,2):w_3(x,y)=([y,x,x,x][y,x,y]-[y,x,x,y][y,x,x])[y,x,x][y,x], 
\]
\[
(4,3^2,2):w_4(x,y)=[y,x,x,x]([y,x,x][y,x,y]-[y,x,y][y,x,x])[y,x], 
\]
\[
(3^4):w_5(x,y)=([y,x,x][y,x,y]-[y,x,y][y,x,x])[y,x,x][y,x,x],
\]
\[
(4,2^4):w_6(x,y)=[y,x,x,x][y,x]^4,
\]
\[
(3^2,2^3):w_7(x,y)=[y,x,x]^2[y,x]^3. 
\]

{\rm (ii)} The following are highest weight vectors for $\lambda=(6,6)$:
\[
(4^2,2^2):w_1(x,y)=([y,x,x,x][y,x,y,y]-[y,x,y,x][y,x,x,y]-
\]
\[
-[y,x,x,y][y,x,y,x]+[y,x,y,y][y,x,x,x])[y,x]^2,
\]
\[
(4,3^2,2):w_2(x,y)=([y,x,x,x][y,x,y][y,x,y]-[y,x,y,x][y,x,x][y,x,y]-
\]
\[
-[y,x,x,y][y,x,y][y,x,x]+[y,x,y,y][y,x,x][y,x,x])[y,x],
\]
\[
(3^4):w_3(x,y)=([y,x,x][y,x,y]-[y,x,y][y,x,x])^2,
\]
\[
w_4(x,y)=[y,x,x][y,x,x][y,x,y][y,x,y]-[y,x,y][y,x,x][y,x,x][y,x,y]-
\]
\[
-[y,x,x][y,x,y][y,x,y][y,x,x]+[y,x,y][y,x,y][y,x,x][y,x,x],
\]
\[
(3^2,2^3):w_5(x,y)=([y,x,x][y,x,y]-[y,x,y][y,x,x])[y,x]^3,
\]
\[
(2^6):w_6(x,y)=[y,x]^6.
\]
\end{lemma}

\begin{proposition}\label{no central polynomials of degree 12}
The algebra $M_4(K)$ does not have central polynomials and polynomial identities of degree $12$ in $K\langle x,y\rangle$.
\end{proposition}

\begin{proof}
Again, we follow the main steps of the proof of Proposition \ref{no central polynomials of degree 10}.
It is sufficient to consider the case $\lambda=(7,5)$ and $\lambda=(6,6)$.

(i) The case $\lambda=(7,5)$. For
\[
(m_1,\ldots,m_k)=(4^2,2^2),(3^2,2^3),(2^6)
\]
there is only one $GL_2(K)$-submodule $W(7,5)$ in the products of commutators
\[
\text{\rm Com}_{m_1}\cdots\text{\rm Com}_{m_k},m_1+\cdots+m_k=12,k\geq 4.
\]
The number of direct summands $W(7,5)$ is equal to 2 in the cases
\[
(m_1,\ldots,m_k)=(4,3^2,2),(3^4).
\]
The corresponding highest weight vectors are given in Lemma \ref{HWV of degree 12} (i).
Depending on the positions of the commutators of length $m_1,\ldots,m_k$ the number of linearly independent highest weight vectors is:
\begin{itemize}
\item $(5,3,2^2)$, $w_1(x,y)$: 12,
\item $(4^2,2^2)$, $w_2(x,y)$: 6,
\item $(4,3^2,2)$, $w_3(x,y),w_4(x,y)$: 24,
\item $(3^4)$, $w_5(x,y),w_4(x,y)$: 3,
\item $(4,2^5)$, $w_6(x,y)$: 5,
\item $(3^2,2^3)$, $w_7(x,y)$: 10.
\end{itemize}
Totally we have 60 linearly independent highest weight vectors.
We consider
\[
w_{(7,5)}(x,y)=\sum_{i=1}^{60}\xi_iw^{(i)}(x,y)
\]
with unknown coefficients $\xi_i\in K$, $i=1,\ldots,60$. As in Proposition \ref{no central polynomials of degree 10}
we replace $x$ by the diagonal generic matrix $u$. Now we replace $y$ by the matrix
\[
v_2=\left(\begin{matrix}1&1&v_{13}&1\\
0&1&1&1\\
1&1&1&1\\
1&1&0&0\\
\end{matrix}\right).
\]
We consider the entry $z_{12}$ of
\[
w_{(7,5)}(x,y)=\sum_{a,b=1}^4z_{ab}E_{ab}.
\]
It is in the form
\[
z_{12}=\sum_{i,j}f_{ij}(\xi_1,\ldots,\xi_{60})u_1^{i_1}u_2^{i_2}u_3^{i_3}u_4^{i_4}v_{13}^j.
\]
We have computed that the homogeneous linear system consisting of all
\[
f_{ij}(\xi_1,\ldots,\xi_{60})=0
\]
has only the zero solution and this shows that $M_4(K)$ does not have polynomial identities and central polynomials for $\lambda=(7,5)$.

(ii) The case $\lambda=(6,6)$. As in the case of the proof of (i), the corresponding highest weight vectors are given in Lemma \ref{HWV of degree 12} (ii).
Depending on the positions of the commutators of length $m_1,\ldots,m_k$ in the product of $\text{\rm Com}_{m_1},\ldots,\text{\rm Com}_{m_k}$,
the number of linearly independent highest weight vectors is:
\begin{itemize}
\item $(4^2,2^2)$, $w_1(x,y)$: 6,
\item $(4,3^2,2)$, $w_2(x,y)$: 12,
\item $(3^4)$, $w_3(x,y),w_4(x,y)$: 2,
\item $(3^2,2^3)$, $w_5(x,y)$: 10,
\item $(2^5)$: $w_6(x,y)$: 1.
\end{itemize}
Totally we have 31 linearly independent highest weight vectors. As in (i), we consider the linear combination $w_{(6,6)}(x,y)$ of these 31 polynomials,
evaluate it for $x=u$, $y=v_2$ and solve the homogeneous linear system obtained from the entry $z_{12}$. The system has only the trivial solution which shows
that $M_4(K)$ does not have polynomial identities and central polynomials for $\lambda=(6,6)$.
The same result can be obtained if we use the entry $z_{12}$ of the evaluation $w_{(6,6)}(u,v_3)$ instead of the evaluation $w_{(6,6)}(u,v_2)$, where
\[
v_3=\left(\begin{matrix}0&1&v_{13}&0\\
0&1&1&0\\
0&1&0&1\\
1&0&0&0\\
\end{matrix}\right).
\]
\end{proof}

As a conclusion, the method based on the combination of representation theory of $GL_d(K)$ and the theory of PI-algebras
is much better from computational point of view than the other considered methods. For example, in the case of polynomials in two variables of degree 12
we have to solve only two systems with 60 and 31 unknowns only. In the corresponding cases based only on representation theory of $GL_d(K)$
we consider, respectively, systems with 297 and 132 unknowns.

\end{document}